\documentclass{article} 
\usepackage{authblk}

\usepackage{amsthm}
\usepackage{amsmath}
\usepackage{amssymb} 
\usepackage{bbold}
\usepackage[hidelinks]{hyperref}
\usepackage{caption, cleveref} 
\usepackage{xcolor}

\usepackage{tikz}
\usetikzlibrary{positioning}
\usetikzlibrary{decorations.pathreplacing}

\newcommand{\A}{\mathcal{A}}

\newcommand{\N}{\mathbb{N}}

\newcommand{\F}{\mathcal{F}}
\renewcommand{\H}{\mathcal{H}}
\renewcommand{\S}{\mathcal{S}}
\newcommand{\HH}{\mathbb{H}}
\renewcommand{\SS}{\mathbb{S}}
\newcommand{\ZZ}{\mathbb{Z}}

\newcommand{\PP}{\mathcal{P}}

\newcommand{\intr}{\mathrm{int}}
\newcommand{\ex}{\mathrm{ex}}
\newcommand{\rk}{\mathrm{rk}}

\newtheorem{theorem}{Theorem}
\newtheorem{lemma}{Lemma}

\newtheorem{proposition}{Proposition}

\title{A topological proof of Ky Fan's covering lemma}
\author{Bogdan Chornomaz\footnote{Department of Mathematics, Technion, Israel. \href{mailto:markyz.karabas@gmail.com}{\bf markyz.karabas@gmail.com}. Funded by the European Union (ERC, GENERALIZATION, 101039692). Views and opinions expressed are, however, those of the author(s) only and do not necessarily reflect those of the European Union or the European Research Council Executive Agency. Neither the European Union nor the granting authority can be held responsible for them.}}
\date{July 29, 2025}

\begin{document}

\maketitle

\begin{abstract}
    We give an intuitive combinatorial proof of Ky Fan's covering lemma based on the Borsuk-Ulam theorem. We then show how this approach can be generalized to Ky Fan's covering lemma for several linear orders. 
\end{abstract}

\section{Ky Fan's covering lemma}\label{sec-ky-fan}
A famous Borsuk-Ulam (BU) theorem from algebraic topology has two counterparts, well-known to be equivalent to it: Tucker's lemma in combinatorics and Lusternik–Schnirelmann (LS) theorem in terms of covers. Ky Fan's (KF) theorem (main theorem in~\cite{kyfan52}) is a strengthening of Tucker's lemma, also stated and proved in combinatorial terms.\footnote{See also~\cite{nyman13} for a further discussion about the landscape around these results.} 
In the same paper, as a corollary, Ky Fan proved a closed cover version of this theorem (KFCL\footnote{Here we tweaked our abbreviation scheme in order to avoid unnecessary parallel with Kentucky Fried Chicken, which the reader probably did not have in the first place before reading this footnote.}), noting that, as a special case, it implies the LS theorem: 
\begin{theorem}[Ky Fan's covering lemma, Theorem~2 in~\cite{kyfan52}]\label{th-kyfan-covers}
	Let $\F$ be a finite antipodal-free closed cover of the $n$-dimensional sphere $\SS^n$, and let $\leq$ be a linear order on $\F$. Then there is $x\in \SS^n$ witnessing an alternating pattern of length $n+2$ for $\F$ and $\leq$, that is, such that there are $F_1 < \dots < F_{n+2} \in \F$ such that $x\in (-1)^{i-1}F_i$, for $i=1, \dots, n+2$. 
\end{theorem}

We do not know if KFCL easily implies his original combinatorial statement, although we suspect that it does not. However, KFCL is perhaps more accessible and is usually sufficient for most applications in combinatorics; see, for example, \cite{simonyi2006local}.\footnote{As a downside, let us note that there is an interest in having constructive proofs of lemmas in combinatorial topology (see~\cite{lz06,prescott05}), while our proof is a step in the opposite direction.}

Here, we give an easy proof of KFCL (\Cref{th-kyfan-covers}) via BU. Apart from a certain didactic value, its proof follows a scheme, explicated afterwards in \Cref{lem-Ky-Fan-generalized}, which we find interesting in its own right. We then use this scheme to prove an asymptotically tight generalization of KFCL for several linear orders (Theorems~\ref{th-kyfan-covers2} and~\ref{th-kyfan-covers3}). We note that Ky Fan's result is one of the standard tools in topological combinatorics that has been around since 50s, and, naturally, there have been numerous generalizations. Most of those, however, are concentrated around either generalizing to $\ZZ_q$-actions~\cite{meunier06,ziegler02}, or to more complicated simplicial complexes~\cite{sarkaria90,zivaljevic10}. The direction we take is, to our knowledge, quite orthogonal to those, and we are not aware of a close counterpart.

\begin{proof}[Proof (Of KFCL via BU)]
	Let us start by briefly explaining the intuition. We say that $u\in \SS^n$ \emph{witnesses} $[k]$ if it witnesses an alternating pattern of length $k$, as defined in the statement of the theorem. That is, if there are $F_1 < \dots < F_{k} \in \F$ such that $u\in (-1)^{i-1}F_i$. Similarly, we say that $u$ \emph{witnesses} $[-k]$ if $-u$ witnesses $[k]$, alternatively, if there are $F_1 < \dots < F_{k} \in \F$ such that $u\in (-1)^{i}F_i$. The following observations are trivial: 
	\begin{itemize}
		\item any $u\in\SS^n$ witnesses either $[2]$ or $[-2]$;
		\item if $u\in\SS^n$ witnesses $[k]$ or $[-k]$ then it witnesses $[l]$ and $[-l]$ for all $2\leq l < k$;
		\item if $u\in\SS^n$ witnesses $[k]$ and $[-k]$ then it witnesses either $[k+1]$ or $[-(k+1)]$;
		\item if $u\in \SS^n$ witnesses $[k]$, then $-u$ witnesses $[-k]$.
	\end{itemize}
	From the above, it is obvious that for every $u$ there is a unique maximal alternating pattern $p(u)$, of the form  $[\pm l]$, witnessed by $u$; here maximality is in the sense of the Hasse diagram $\H_{2, \infty}$ in \Cref{fig-hasse} below. Moreover, $\H_{2, \infty}$ is equipped with an order-preserving involution $[l] \mapsto [-l]$ that makes $p$ equivariant, that is, such that, trivially, $p(-u) = -p(u)$.
	
	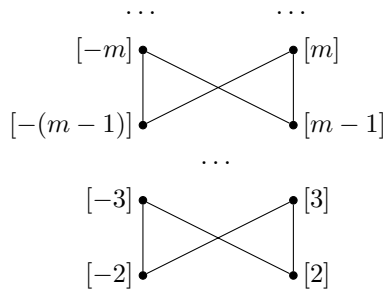
\begin{figure}[hbt]
	    \centering
		\begin{tikzpicture}
	[bpoint/.style={inner sep = 1.0pt, circle,draw,fill=black},
	wpoint/.style={inner sep = 1.7pt, circle,draw,fill=white}]

	\node[bpoint](-2) at (-1, 0) {};
	\node[bpoint](2) at (1, 0) {};
	\node[bpoint](-3) at (-1, 1) {};
	\node[bpoint](3) at (1, 1) {};
	\node[bpoint](-m-1) at (-1, 2) {};
	\node[bpoint](m-1) at (1, 2) {};
	\node[bpoint](-m) at (-1, 3) {};
	\node[bpoint](m) at (1, 3) {};
	\node at (0, 1.5) {$\dots$};

	\node at (-1, 3.5) {$\dots$};
	\node at (1, 3.5) {$\dots$};

	\node[left] at (-2) {$[-2]$};
	\node[right] at (2) {$[2]$};
	\node[left] at (-3) {$[-3]$};
	\node[right] at (3) {$[3]$};
	\node[left] at (-m-1) {$[-(m-1)]$};
	\node[right] at (m-1) {$[m-1]$};
	\node[left] at (-m) {$[-m]$};
	\node[right] at (m) {$[m]$};
	
	\draw(-2)--(-3)--(2)--(3)--(-2);
	\draw(-m-1)--(-m)--(m-1)--(m)--(-m-1);
\end{tikzpicture}
	    \caption{Hasse diagram $\H_{2, \infty}$ of partial order of alternating patterns.}
	    \label{fig-hasse}
	\end{figure}
	
	The idea of the proof is then to ``smoothify'' the discrete map $p\colon \SS^n \rightarrow \H_{2, \infty}$ into a continuous equivariant map $\varphi$ to the chain complex $\HH^m = C(\H_{2, m})$, where $\H_{2, m}$ is the order-ideal of elements of $\H_{2, m}$ below $[\pm m]$. Using the well-known fact that $\HH^m$ is homeomorphic to $\SS^{m-2}$, the standard BU theorem then implies $m\geq n+2$. We will now work through the details of this scheme.
	
	Let us define a \emph{sample}\footnote{To an extent, the terminology is derived from that of learning theory; the relevant textbook would be~\cite{shays14}. However, the parallel is only by analogy.} as a function from $\F$ to $\{-1, 1, *\}$, where $*$ is treated as ``undefined''. The set of samples is naturally partially ordered by extension, that is, $s_1 \leq s_2$ if for any $F\in \F$, $s_1(F) = s_2(F)$ whenever $s_1(F)\neq *$; in this case, we say that $s_1$ is a \emph{subsample} of $s_2$.  For $u\in \SS^n$, the sample $s(u)$ of $s$ is a characteristic function of $u$ belonging to $F\in \F$, that is, $s(u)(F) = 1$ if $u\in F$, $-1$ if $u\in -F$, and $*$ otherwise. Note that as every $F\in \F$ is antipodal-free, this is well-defined. Each sample trivially maps to a unique maximal alternating pattern that it witnesses, and we denote this map from samples to patterns by~$\rho$. Note that, formally, for an arbitrary pattern $s$, $\rho(s)$ can be $[0]$ or any $[\pm k]$ for $k\in \{1, \dots, |\F|\}$. However, by the above properties of the covers, for $u\in \SS^n$, $\rho\circ s(u) \in \H_{2, \infty}$, that is, $\rho\circ s(u) = [\pm k]$ for $k \geq 2$.
	
	It will also be convenient to identify a sample $s$ with a \emph{consistent} subset $s' \subseteq \F\times \{\pm 1\}$, where consistent means that for every $F\in \F$, at most one of $(F, -1)$ and $(F, 1)$ is in $s'$. Formally, $s' = \{(F, y)\in \F\times \{\pm 1\} ~|~ s(F) = y\}$ and, in the other direction, $s(F) = +1$ if $(F, +1)\in s'$, $-1$ if $(F, -1)\in s'$, and $*$ otherwise. Note that in terms of $s'$, the partial order on samples is simply by inclusion. Further on, with an abuse of notation, we identify thus defined $s$ and~$s'$.

	The above representation is useful as it enables us to define a continuous ``sample function'' on $\SS^n$. Let us fix $\varepsilon > 0$ and, for $u\in \SS^n$, let us define a function $\chi_u\colon \F\times \{\pm 1\} \rightarrow [0, 1]$ as 
	\begin{align*}
		\chi_u(F, y) = \begin{cases}
				0, & d(x, yF) \geq \varepsilon \\
				1 - d(x, yF)/\varepsilon, & 0\leq d(x, yF)\leq \varepsilon.
			\end{cases}
	\end{align*}
	Alternatively, $\chi_u(F, y) = \max(1 - d(x, yF)/\varepsilon, 0)$. Note also that $s(u) = \{(F, y)~|~\allowbreak\chi_u(F, y) = 1\}$. Now, trivially, $\chi_u$ can be canonically decomposed as $\chi_u = \alpha_1 s_1 + \dots \alpha_q s_q$ for $s(u) = s_1 < s_2 < \dots < s_q$ with $\alpha_1 + \dots + \alpha_q = 1$. Here we identify a sample $s_i$ with its characteristic function on $\F\times \{\pm1\}$. This is illustrated in \Cref{fig-decomposition} below.

	\begin{figure}[hbt]
	    \centering
		\begin{tikzpicture}
	[bpoint/.style={inner sep = 1.5pt, circle,draw,fill=black},
	wpoint/.style={inner sep = 1.5pt, circle,draw,fill=white}]

	\draw [decorate,decoration={brace,amplitude=5pt,raise=-4ex}]
  		(2.75,1.15) -- (4.25,1.15) node[midway,yshift=-0.25em]{$s(u)$};
  
	\begin{scope}
		\node at (0,0) {\large$\chi_u:$};
		\draw (0.5, 0)--(5.5, 0);
		\node[wpoint] (1) at (1,0) {};
		\node[bpoint] (2) at (2,0) {};
		\node[bpoint] (3) at (3,0) {};
		\node[wpoint] (4) at (4,0) {};
		\node[bpoint] (5) at (5,0) {};
		
		\node[above] at (1) {$0.5$};
		\node[above] at (2) {$0.3$};
		\node[above] at (3) {$1$};
		\node[above] at (4) {$1$};
		\node[above] at (5) {$0.9$};
	\end{scope}
	\node at (3, -0.5) {\Huge$=$};

	\begin{scope}[yshift=-1cm]
		\draw (0.5, 0)--(5.5, 0);
		\node at (0,0) {$0.3$};
		\node[wpoint] (1) at (1,0) {};
		\node[bpoint] (2) at (2,0) {};
		\node[bpoint] (3) at (3,0) {};
		\node[wpoint] (4) at (4,0) {};
		\node[bpoint] (5) at (5,0) {};
	\end{scope}
	\begin{scope}[yshift=-1.5cm]
		\draw (0.5, 0)--(5.5, 0);
		\node at (-0.5,0) {$+$};
		\node at (0,0) {$0.2$};
		\node[wpoint] (1) at (1,0) {};
		\node[bpoint] (3) at (3,0) {};
		\node[wpoint] (4) at (4,0) {};
		\node[bpoint] (5) at (5,0) {};
	\end{scope}
	\begin{scope}[yshift=-2cm]
		\draw (0.5, 0)--(5.5, 0);
		\node at (-0.5,0) {$+$};
		\node at (0,0) {$0.4$};
		\node[bpoint] (3) at (3,0) {};
		\node[wpoint] (4) at (4,0) {};
		\node[bpoint] (5) at (5,0) {};
	\end{scope}
	\begin{scope}[yshift=-2.5cm]
		\draw (0.5, 0)--(5.5, 0);
		\node at (-0.5,0) {$+$};
		\node at (0,0) {$0.1$};
		\node[bpoint] (3) at (3,0) {};
		\node[wpoint] (4) at (4,0) {};
	\end{scope}
	\node at (3, -3.25) {\Huge$\xrightarrow{~~~~~~\rho~~~~~~}$};
	\node at (3, -4) {$0.3[4] + 0.2[4] + 0.4[-3] + 0.1[-2]$};
	
\end{tikzpicture}
	    \caption{Decomposition of $\chi_u$.}
	    \label{fig-decomposition}
	\end{figure}
	
	The final part of the argument relies on the following claim. Its proof is a simple topological exercise (although let us refer to Section~3.1 in~\cite{chase2024local}). \underline{Claim}: \emph{For $\varepsilon$ small enough\footnote{Here $\varepsilon$ depends on $\F$, but not on $u\in \SS^n$.}, for every $u\in \SS^n$ there is $v\in \SS^n$ such that the support $s_\varepsilon(u)$ of $\chi_u$ is a subsample of $s(v)$.} 
	
	Note that fixing $\varepsilon$ as in the claim makes $s_\varepsilon(u)$ consistent for all $u\in \SS^n$. Thus, if we denote the poset of samples by $\S_\F$, then $u\mapsto \chi_u$ defines a continuous equivariant map from $\SS^n$ to the chain complex $C(\S_\F)$, where the antipodality on $\S_\F$ is by flipping the signs in the sample; as we do not allow an empty sample, this antipodality is fixed-point free. 
	
	Let now $m$ be such that for every $u\in \SS^n$, $\rho \circ s(u) =[\pm k]$ for $k\leq m$. Alternatively, this can be stated as $\rho \circ s(u)\in \H_{2, m}$. The claim then implies that $u\mapsto \rho \circ \chi_u$, with $\rho$ extended by linearity, is a continuous equivariant map from $\SS^n$ to $\HH^m = C(\H_{2,m}) \cong \SS^{m-2}$, finishing the proof.
\end{proof}

\section{Abstracting the approach}\label{sec-abstracting}

Note that the above construction of the continuous map $u\mapsto \chi_u$ from $\SS^n$ to $C(\S_\F)$ is rather universal. We only rely on the fact that the sets in $\F$ are closed and antipodal free, which makes this map to be well-defined\footnote{We also can relax the condition that $\F$ is a cover in order for it to omit the empty sample, and only require that $\F \cup -\F$ is a cover; however, we stick with $\F$ being a cover for consistency.}. 
The specifics of the statement then come from the map $\rho$ that maps samples to alternation patterns. The conditions that $\rho$ is equivariant and that $C(\H_{2, m})$ is a sphere are rather restrictive and can be relaxed: To use BU, it is enough that $\rho$, when lifted to the chain complex, does not collapse antipodal points, which we can require even when the chain complex of the target poset is not equipped with antipodality. And, if we are willing to lose a factor of $1/2$, we can only track the dimension of the target. Let us now explicate some terminology that would enable us to state a certain generalization of \Cref{th-kyfan-covers} above.

For a domain $X$, which we assume to be finite, we define a \emph{sample}~$s$ over $X$ as a partial function from $X$ to $\{\pm1\}$; formally, $s\colon X\rightarrow \{-1, 1, *\}$, where~$*$ is treated as ``undefined''. We denote the sample with an empty domain by~$0$. The support of $s$ is a subset of $X$ where $s\neq *$. We define a poset~$\S^0_X$, or simply $\S^0$ if the domain $X$ is clear from the context, as the set of all samples over~$X$ ordered by extension. The antipodality map $s\mapsto -s$ on $\S^0$ is defined by flipping the signs of values of $s$, that is, $-s(x) = -s(x)$ whenever $s(x)\in \{\pm1\}$, and $-s(x)=*$ whenever $s(x)=*$. We define $\S$ as the poset of all nonempty samples, that is, $\S = \S^0 - 0$. 
Trivially, the antipodal map on $\S$ is fixed-point free.

From now on, by an \emph{$n$-cover} we will mean a finite, closed, antipodal-free cover of $\SS^n$. As before, for a fixed $n$-cover $\F$, for $u\in \SS^n$ we define the \emph{sample} $s_u$, or $s(u)$, as $s_u \colon \F\rightarrow \{\pm1, *\}$ where $s_u(F) = 1$ if $u\in F$, $-1$ if $u\in -F$, and $*$ otherwise. Note that, by the conditions on the $n$-cover, no $u$ belongs to both~$F$ and $-F$, and so the definition is consistent. Moreover, $s(u)\neq 0$, in other words, $s(u)\in \S_\F$. 

We now define $H^\ex_\F\subseteq \S_\F$ as $H^\ex_\F = \{s(u)~|~u\in \SS^n\}$. Moreover, let us define $H^\intr(\F)$ to be the interval closure of $H^\ex(\F)$, that is, $H^\intr_\F = \{s\in \S_\F~|~\text{there are } s', s''\in H^\ex_\F \text{ s.t. } s'\leq s\leq s''\}$. Notationwise, the superscripts $\ex$ and $\intr$ stand for ``exact'' and ``interval''; as usual, we omit $\F$ in the subscript whenever it is clear from the context. Trivially, $H^\ex \subseteq H^\intr$. Note that, by construction, $s(-u) = -s(u)$, and so all $H^\ex$ and $H^\intr$ are antipodal closed.

Recall that, for a poset $\PP$, the rank $\rk(p)$ of an element $p\in \PP$ is the maximal length of a strictly increasing chain to $p$, where the length of a chain is the number of its elements minus one. 
For example, the rank of the elements $[\pm m]$ in $\H_{2, \infty}$  in \Cref{fig-hasse} is $m-2$. 

\begin{lemma}\label{lem-Ky-Fan-generalized}
	For an $n$-cover $\F$, let $\rho\colon H^\intr_\F\rightarrow \PP$ be a monotone map into a poset $\PP$ such that $\rho(s) \neq \rho(-s)$ for all $s\in H^\intr_\F$. Then there is $u\in \SS^n$ such that $\rho\circ s(u)\in \PP$ has rank at least $(n + 1)/2$. 
	
	Furthermore, if $\PP = \H_{2, \infty}$, then there is $u\in \SS^n$ such that the rank of $\rho\circ s(u)\in \PP$ is at least $n$.
\end{lemma}
We note that the information about the witnessed patterns relies on a particular choice of $\PP$ and $\rho$. For example, in the proof of \Cref{th-kyfan-covers}, $\PP$ was the poset of alternation patterns $\H_{2, \infty}$. The conclusion of \Cref{lem-Ky-Fan-generalized} then yielded $u\in \SS^n$ sitting sufficiently high in $\PP$, and thus witnessing sufficiently high alternating pattern. Alternatively, picking $\PP$ to be the poset of samples $\S$ itself would guarantee the existence of a point in $\SS^n$ that belongs to sufficiently many sets in the cover or their antipodals.

The second part of the claim relies on the knowledge of the special structure of $C(\H_{2, m})$; namely, that it is an $(m-2)$-sphere. We keep it mostly for consistency with \Cref{th-kyfan-covers}, and note additionally that it can be further clarified in terms of the coindex of $C(\PP)$. The bottom line, however, is that we can completely ignore its structure and only lose a factor of $2$.
\begin{proof}
	Let $s(u)$, $s_\varepsilon(u)$, and the map $u\mapsto \chi_u$ be exactly as in the proof of \Cref{th-kyfan-covers}. Recall that we argued there that $\chi_u$ is a formal convex combination of a chain $\zeta$ of samples in $S_\F$ between $s(u)$ and $s_\varepsilon(u) = \sup(\chi_u)$. Moreover, for $\varepsilon$ small enough, for every $u\in \SS^n$ there is $v\in \SS^n$ such that $s_\varepsilon(u) \leq s(v)$. Thus, $\zeta$ is a chain in $H^\intr_\F$, so $u\mapsto \rho(\chi_u)$ is a continuous map from $\SS^n$ to $C(\PP)$.
	
	It is easy to see that $-s_\varepsilon(u) = s_\varepsilon(-u)$. Monotonicity of $\rho$, together with the fact that $\rho\circ s_\varepsilon(u) \neq \rho\circ s_\varepsilon(-u)$, imply that either $\rho(t) \neq \rho \circ s_\varepsilon(u)$ for all $t\leq s_\varepsilon(-u)$ or, the other way round, $\rho(t) \neq \rho\circ s_\varepsilon(-u)$ for all $t\leq s_\varepsilon(u)$. This implies that the map $u\mapsto \rho(\chi_u)$ does not collapse antipodal points.
	
	If the maximal rank of $\rho \circ s_\varepsilon(u)$ is $m$, then the map $u\mapsto \rho(\chi_u)$ is actually to $C(\PP^{\leq m})$, where $\PP^{\leq m}$ is the subposet of $\PP$ of the elements of rank at most $m$. Trivially, the dimension of $C(\PP^{\leq m})$ is at most $m$. By~\cite{Jaw02}, this implies $2m \geq n+1$ (see also Theorem 3.1 in~\cite{chase2024local} for a statement more adjusted to our case). The proof is finished by taking, for $u$ witnessing $\rk \circ \rho \circ s_\varepsilon(u) = m$, $v\in \SS^n$ for which $s_\varepsilon(u) \leq s(v)$. Then $\rk \circ \rho \circ s(v) \geq \rk \circ \rho \circ s_\varepsilon(u) \geq m \geq (n+1)/2$, as needed.
	
	The second claim follows by observing that for $\PP = \H_{2, \infty}$, the poset $\PP^{\leq m} = \H_{2, m+2}$, and $C(\H_{2,m+2})$ is not only $m$-dimensional, but is actually an $m$-sphere. Hence, instead of~\cite{Jaw02}, we can use a regular Borsuk-Ulam to imply that ${m\geq n}$. 
\end{proof}

\section{Ky Fan's for several orders}\label{sec-several-orders}
After isolating the abstract core of the proof, it would be nice to illustrate its utility by, perhaps, proving something new along these lines, potentially involving some cool target posets. Theorems~\ref{th-kyfan-covers2} and~\ref{th-kyfan-covers3} below do just that by generalizing KFCL to several linear orders. Although we consider these results to be nice and neat, we have to admit that our primary goal was to find such illustrations of utility. And our main remaining question, if there is any, is whether something else can be milked out of this setup.\footnote{The author entertains the possibility of relating it to something like sample compression conjecture~\cite{fw95,lw86}, however all parallels do not go beyond some similarity in the vibe.}

The proof is split into two parts: \Cref{th-kyfan-covers2}, generalizing it to two orders, is followed by \Cref{th-kyfan-covers3}, which further generalizes it to an arbitrary number of orders. However, all proof ideas are already present in the two orders case, and further generalization is straightforward, even if somewhat technical.

\begin{theorem}[Ky Fan's covering lemma for two orders]\label{th-kyfan-covers2}
	For an $n$-cover $\F$, let $\leq_1$ and $\leq_2$ be two linear orders on $\F$. Then there is $u\in \SS^n$ witnessing an alternating monotone pattern of length at least $m = \left((n+1)/2\right)^{1/2}$ for both of those orders. 
	
	That is, there is $u\in \SS^n$ and $F_1, \dots, F_m\in \F$ such that:
	\begin{itemize}
		\item Either $u\in (-1)^{i-1}F_i$, for $i=1, \dots, m$ or $u\in (-1)^{i}F_i$, for $i=1, \dots, m$;
		\item Either $F_1\leq_1 F_2 \leq_1 \dots \leq_1 F_m$, or $F_1\geq_1 F_2 \geq_1 \dots \geq_1 F_m$, and similarly for $\leq_2$.
	\end{itemize}
    
    Moreover, this is asymptotically sharp: there is an $n$-cover such that the largest alternating monotone pattern with respect to it has length at most $\lceil\sqrt{n+2}\rceil$.
\end{theorem}
In fact, we prove a slightly more general statement that if $h_{++}$ is the length of alternating pattern increasing in both orders, and $h_{+-}$ increasing in $\leq_1$ and decreasing in $\leq_2$, then there is $u$ witnessing $h_{++} \cdot h_{+-} \geq (n+1)/2$.

The asymptotical sharpness $m\sim \sqrt{n}$ is rather trivial: If we pick an arbitrary antipodal-free cover of $\S^n$ with $n+2$ sets (a minimal number allowed by LS), and ``maximally uncorrelated'' orders $\leq_1$ and $\leq_2$, then any maximal chain, monotonous for both of them (even without the alternation restriction), will have at most $\lceil\sqrt{n+1}\rceil$ elements. The details of this construction, in a multiple orders setting, are explained in \Cref{prop-d-orders} afterwards, and will not be given in this proof.

We finally note that, if faced with the statement of the theorem, it feels natural to try attacking it along the lines of Erd\"os-Szekeres theorem\footnote{It is an appropriate place to thank Fr\'ed\'eric Meunier and Shay Moran for helpful discussions. Notably, both of them had a similar instinct to make a move in this direction.}. However, we have not been able to achieve anything by naively applying or mimicking it here.
\begin{proof}
	Let us fix a ``realizable'' sample $s\in \S_\F$, where by realizable we mean that $s$ is nonempty, moreover, contains $(F_1, -)$ and $(F_2, +)$ for some $F_1, F_2\in \F$. Note that as $\F$ is a cover, any sample in $H^\intr_\F$ is trivially realizable. 
	 
	Now, for $F$ in the support of $s$, let us define $\iota_{++}(F) \in \N$ as the length of the maximal alternating, with respect to $s$, pattern, which is increasing in both $\leq_1$ and $\leq_2$ and ends in $F$. That is, $\iota_{++}(F)$ is the maximal $k$ such that there is $F_1, \dots, F_k \in \sup(s)$ such that:
	\begin{itemize}
		\item $s(F_i) = -s(F_{i-1})$ for $i=1, \dots, k-1$;
		\item $F_i \leq_1 F_{i+1}$ and $F_i \leq_2 F_{i+1}$  for $i=1, \dots, k-1$;
		\item $F_k = F$.
	\end{itemize}
	
	We define $\iota_{+-}(F)$ similarly, except for requiring that the respective pattern is \emph{decreasing} in $\leq 2$. Finally, we define $\iota \colon \sup (s) \rightarrow \N^2$ as $\iota(F) = \left(\iota_{++}(F), \iota_{+-}(F)\right)$. This is illustrated in \Cref{fig-iotas} below.
	
	\begin{figure}[hbt]
	\centering
	\begin{tikzpicture}
	[bpoint/.style={inner sep = 1.5pt, circle,draw,fill=black},
	wpoint/.style={inner sep = 1.5pt, circle,draw,fill=white}]

	\begin{scope}[scale=0.7]
	\draw[->](-0.5, 0)--(6.5, 0);
	\draw[->](0, -0.5)--(0, 6.5);
	\node at (6.5, -0.5) {$\leq_1$};
	\node at (-0.5, 6.5) {$\leq_2$};
	
	\node[wpoint] (1) at (1, 3) {};
	\node[bpoint] (2) at (2, 1) {};
	\node[wpoint] (3) at (3, 5) {};
	\node[bpoint] (4) at (4, 2) {};
	\node[wpoint] (5) at (5, 6) {};
	\node[bpoint] (6) at (6, 4) {};
	
	\draw[dotted] (0, 3)--(1)--(1, 0);
	\draw[dotted] (0, 1)--(2)--(2, 0);
	\draw[dotted] (0, 5)--(3)--(3, 0);
	\draw[dotted] (0, 2)--(4)--(4, 0);
	\draw[dotted] (0, 6)--(5)--(5, 0);
	\draw[dotted] (0, 4)--(6)--(6, 0);
	
	\node[above] at (1) {$(1,1)$};
	\node[above] at (2) {$(1,2)$};
	\node[above] at (3) {$(2,1)$};
	\node[above] at (4) {$(1,2)$};
	\node[above] at (5) {$(2,1)$};
	\node[above] at (6) {$(2,2)$};
	
	\end{scope}

\end{tikzpicture}
	\caption{Function $\iota$ for a given sample. The points are $F\in \sup(s)$, colored with respect to $s(F)$.}
	\label{fig-iotas}
	\end{figure}
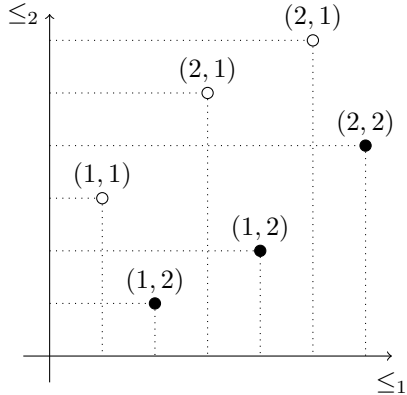
	
	Let us make the following easy observation: 
	\begin{equation}\label{obs-iota}\tag{*}
	\text{\emph{For every $F, G \in \sup(s)$, if $s(F)\neq s(G)$, then $\iota(F)\neq \iota(G)$.}}
	\end{equation} Indeed, for $F\neq G$, one of the four mutually exclusive situations hold: i) $F <_1 G$, $F <_2 G$; ii) $F <_1 G$, $F >_2 G$; iii) $F >_1 G$, $F <_2 G$; or iv) $F >_1 G$, $F >_2 G$. If $s(F)\neq s(G)$ than these situations imply that i) $\iota_{++}(F) < \iota_{++}(G)$; ii) $\iota_{+-}(F) < \iota_{+-}(G)$; iii) $\iota_{+-}(F) > \iota_{+-}(G)$; and iv) $\iota_{++}(F) > \iota_{++}(G)$. In any case, $\iota(F)\neq \iota(G)$, as needed.
	
	Let us treat $\N^2$ as a poset with a usual product order; let $\A(\N^2)$ be the poset of nonempty order-ideals of $\N^2$, where we treat an element of $\A(\N^2)$ as an antichain of its maximal elements. Finally, let $\PP^0$ be a poset of tuples $(a, f)$, where $a\in \A(\N^2)$ is an antichain in $\N^2$ and $f\colon a\rightarrow \{\pm 1\}$ is arbitrary, where the order on $\PP^0$ is defined as $(a, f)\leq (b, g)$ if $a<b$ or $a=b$ and $f=g$. This is illustrated in \Cref{fig-hasse2} below.
	
	\begin{figure}[hbt]
	    \centering
		\begin{tikzpicture}
	[bpoint/.style={inner sep = 1.0pt, circle,draw,fill=black},
	wpoint/.style={inner sep = 1.7pt, circle,draw,fill=white}]

	\begin{scope}[xscale = 0.8]
	\begin{scope}	
			
	\node (11) at (0,0) {$11$};
	\node (12) at (-1,1) {$12$};
	\node (21) at (1,1) {$21$};
	\node (13) at (-2,2) {$13$};
	\node (22) at (0,2) {$22$};
	\node (31) at (2,2) {$31$};
	\node (14) at (-3,3) {$14$};
	\node (23) at (-1,3) {$23$};
	\node (32) at (1,3) {$32$};
	\node (41) at (3,3) {$41$};

	\draw (12)--(11)--(21) (13)--(12)--(22)--(21)--(31)
		(14)--(13)--(23)--(22)--(32)--(31)--(41);

	\node at (0, -1) {$\N^2$};
	\end{scope}
	
	\begin{scope}[xshift=7cm]
	\node (11) at (0,0) {$11$};
	\node (12) at (-1,1) {$12$};
	\node (21) at (1,1) {$21$};
	\node (13) at (-2,2) {$13$};
	\node (1221) at (0,2) {$12,21$};
	\node (31) at (2,2) {$31$};

	\node (22) at (0,3) {$22$};

	\node (14) at (-3,3) {$14$};
	\node (1321) at (-1.25,3) {$13,21$};
	\node (1231) at (1.25,3) {$12,31$};
	\node (41) at (3,3) {$41$};

	\draw (12)--(11)--(21) (13)--(12)--(1221)--(21)--(31)
		(14)--(13)--(1321)--(1221)--(1231)--(31)--(41)
		(1221)--(22);

	\node at (0, -1) {$\A(\N^2)$};
	\end{scope}	
	\end{scope}
	
	\begin{scope}[yshift = -6cm, xshift = 3cm, yscale=2, xscale=1.2]				
	\node (11-) at (-1,0) {$11-$};
	\node (11+) at (1,0) {$11+$};
	\node (12-) at (-2,1) {$12-$};
	\node (12+) at (-1,1) {$12+$};
	\node (21-) at (1,1) {$21-$};
	\node (21+) at (2,1) {$21+$};	
	
	\draw (12-)--(11-)--(12+)--(11+)--(12-)
		 (21-)--(11-)--(21+)--(11+)--(21-);
		 
	\node(12-21-) at (-2.25, 2) {$12-, 21-$};
	\node(12-21+) at (-0.75, 2) {$12-, 21+$};
	\node(12+21-) at (0.75, 2) {$12+, 21-$};
	\node(12+21+) at (2.25, 2) {$12+, 21+$};
	
	\draw (12-21-)--(12-) (12-21-)--(12+) (12-21-)--(21-) (12-21-)--(21+)
	  (12-21+)--(12-) (12-21+)--(12+) (12-21+)--(21-) (12-21+)--(21+)
	  (12+21-)--(12-) (12+21-)--(12+) (12+21-)--(21-) (12+21-)--(21+)
	  (12+21+)--(12-) (12+21+)--(12+) (12+21+)--(21-) (12+21+)--(21+);

	\node(13-) at (-4.75, 2) {$13-$};
	\node(13+) at (-3.75, 2) {$13+$};
	\node(31-) at (3.75, 2) {$31-$};
	\node(31+) at (4.75, 2) {$31+$};

	\draw (13-)--(12-)--(13+)--(12+)--(13-)
		(31-)--(21-)--(31+)--(21+)--(31-);

	\node at (0, -0.5) {$\PP^0$};
	\end{scope}	
\end{tikzpicture}
	    \caption{Posets $\N^2$, $\A(\N^2)$, and $\PP^0$.}
	    \label{fig-hasse2}
	\end{figure}
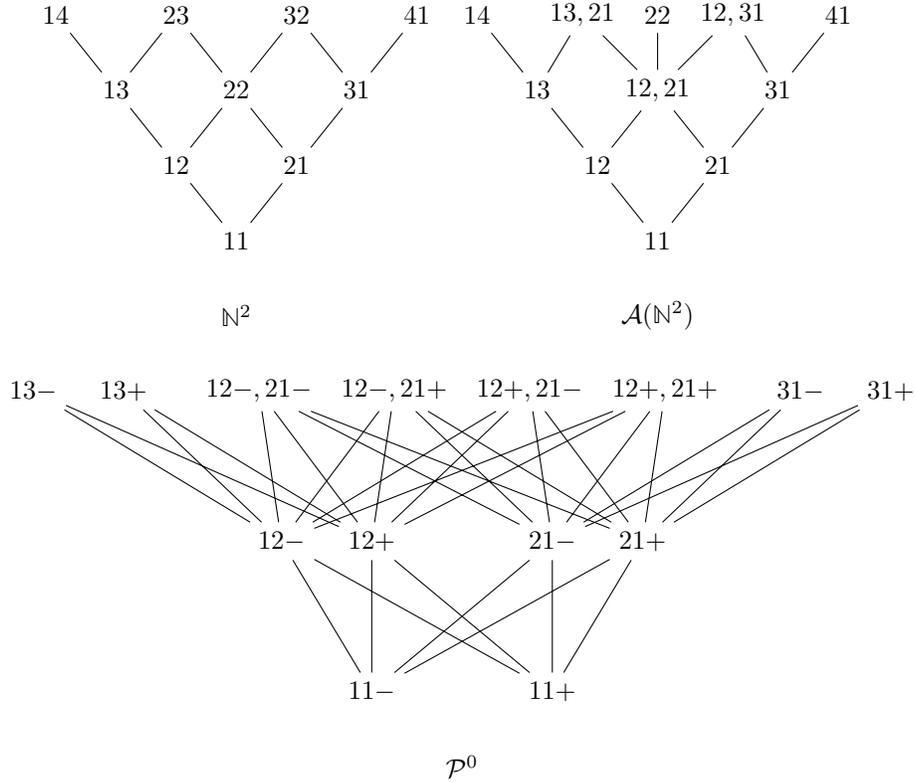
	
	Let us now define $\rho \colon \S_\F \rightarrow \PP^0$ as follows: $\rho(s) = (a_s, f_s)$, where $a_s$ is the antichain of maximal elements in the image $\iota(s)\subseteq \N^2$, and $f_s \colon a_s\rightarrow \{\pm 1\}$ is the coloring of $a_s$ under $s$. That is, for $x \in a_s$, $f_s(x) = -$ if there is $F\in \sup(s)$ such that $\iota(F) = x$ and $s(F) = -$; similarly, $f_s(x) = +$ if there is $F\in \sup(s)$ such that $\iota(F) = x$ and $s(F) = +$. The observation above ensures that $f_s$ is well-defined. For the pattern in \Cref{fig-iotas}, $\iota(s) = \{(1,1), (1,2), (2,1), (2,2)\}$, $a_s = \{(2,2)\}$, and $f_s = (2,2) \mapsto -$. 
	
	We aim at applying \Cref{lem-Ky-Fan-generalized}. First of all, note that, by the remark in the beginning of the proof, we can assume that there are $(F_1, -), (F_2, +) \in s$. This enables us to treat $\rho$ as a map $\rho \colon H^\intr_\F \rightarrow \PP$, where $\PP = \PP^0 - \{(11, -), (11, +)\}$. Doing this is rather optional, but it bumps the rank of the element we get by $1$. Now, trivially, $\rho(s)$ does not collapse antipodal samples: indeed, if $\rho(s) = (a, f)$ then $\rho(-s) = (a, -f) \neq (a, f)$. What is left to check is that $\rho$ is monotone. 
	
	Suppose not, that is, there are samples $s_1 \leq s_2$ such that for $\rho(s_1) = (a_1, f_1)$ and $\rho(s_2) = (a_2, f_2)$, it holds $(a_1, f_1)\not\leq (a_2, f_2)$. First note that $s_1 \leq s_2$ trivially implies $a_1\leq a_2$. If $a_1\lneq a_2$ then, by construction of $\PP$, $(a_1, f_1)\leq (a_2, f_2)$. So, to invalidate the latter, we need to have $a = a_1 = a_2$ and $f_1 \neq f_2$, that is, there is some $x\in a$ such that $f_1(x) \neq f_2(x)$. Let $F_1 \in \sup(s_1)$ and $F_2\in \sup(s_2)$ be the elements in the cover witnessing $x\in a_1$ and $x\in a_2$ respectively, that is, such that $\iota^1(F_1) = \iota^2(F_2) = x$. Note that here $\iota^1$ is the function $\iota$, constructed for the sample $s_1$, and similarly for $s_2$. By the construction of $f_1$ and $f_2$, $s_1(F_1) = f_1(\iota_1(F_1)) = f_1(x) \neq f_2(x) = f_2(\iota_2(F_2)) = s_2(F_2)$.
	
	Now, as $s_1 \leq s_2$, $F_1 \in \sup(s_1)\subseteq \sup(s_2)$. Thus, $\iota^2(F_1)$ is well-defined, moreover, trivially, $\iota^2(F_1)\geq \iota^1(F_1) = x$. However, $x\in a^2$ is maximal in the $\iota^2$-image of $\sup (s_2)$, and so $\iota^2(F_1) = x = \iota^2(F_2)$. But $s_2(F_1) = s_1(F_1) \neq  s_2(F_2)$, contradicting to~\eqref{obs-iota}.

	Now, by \Cref{lem-Ky-Fan-generalized}, we get $u\in \SS^n$ such that for $s = s(u)$ and $w=\rho(s)$, $\rk_{\PP}(w) \geq (n+1)/2$. In terms of $\PP^0$, which is easier to deal with, $\rk_{\PP^0}(w) \geq (n+3)/2$. Also, by construction of $\PP^0$, for $w = (a, f)$, $\rk_{\PP^0}(w) = \rk_{\A(\N^2)}(a) \geq (n+3)/2$. For $a\in \A(\N^2)$, let $h_{++}(a) = \max \{i~|~i,j\in a\}$, and $h_{+-}(a) = \max \{j~|~i,j\in a\}$; note that, once we unpack the definitions, $h_{++}$ is the size of the maximal alternating chain increasing in both $\leq_1$ and $\leq_2$, witnessed by $u$, and similarly for $h_{+-}$. Trivially, $a\leq \{(h_{++}(a), h_{+-}(a))\}$, and hence $\rk_{\A(\N^2)}\left(\{(h_{++}(a), h_{+-}(a))\}\right) = h_{++}(a) \cdot h_{+-}(a) - 1 \geq (n+3)/2$. Thus, $h_{++}(a) \cdot h_{+-}(a) \geq (n+1)/2$. Thus, $\max(h_{++}(a), h_{+-}(a)) \geq \left((n+1)/2\right)^{1/2}$, as needed.
\end{proof}

\begin{theorem}[Ky Fan's covering lemma for multiple orders]\label{th-kyfan-covers3}
	For an $n$-cover $\F$ and $d\geq 1$, let $\leq_i$, for $i=1, \dots, d$ be $d$ linear orders on $\F$. Then there is $u\in \SS^n$ witnessing an alternating pattern of length at least $m = \left((n+1)/2\right)^{1/2^{d-1}}$, monotone with respect to all $d$ of these orders. 

    Moreover, this is asymptotically sharp: there is an $n$-cover such that the largest alternating monotone pattern with respect to it has length at most $\left\lceil(n+2)^{D}\right\rceil$, for $D = 1/2^{d-1}$.
\end{theorem}
\begin{proof}
	The proof follows the same pattern as the one in \Cref{th-kyfan-covers2}. We start by fixing a realizable sample $s\in \S_\F$.
	
	Let us define $R$ as a set of functions $r\colon [d]\rightarrow \{\pm 1\}$ such that $r(1) = 1$; note that $|R| = 2^{d-1}$. Now, for $r\in R$, we define $\iota_r \colon \sup(s) \rightarrow \N$ as the length of the maximal alternating, with respect to $s$, pattern, which is increasing in $\leq_{r(i)\cdot i}$ for all $i\in [d]$, and ends in $F$; here $\leq_{-i}$ is defined as the order, inverse to $\leq_i$.  That is, $\iota_r(F)$ is the maximal $k$ such that there is $F_1, \dots, F_k \in \sup(s)$ such that:
		\begin{itemize}
			\item $s(F_i) = -s(F_{i-1})$ for $i=1, \dots, k-1$;
			\item $F_i \leq_{r(j)\cdot j} F_{i+1}$ for all $i=1, \dots, k-1$ and $j\in [d]$;
			\item $F_k = F$.
		\end{itemize}
	Finally, we define $\iota \colon \sup (s) \rightarrow \N^R$ as $\iota(F) = \left(\iota_r(F)\right)_{r\in R}$. Similarly to how it was before, the following holds:
		\begin{equation}\label{obs-iota2}\tag{*}
		\text{\emph{For every $F, G \in \sup(s)$, if $s(F)\neq s(G)$, then $\iota(F)\neq \iota(G)$.}}
		\end{equation}	
		
	Again, similarly, we consider posets $\N^R$, $\A(\N^R)$, and $\PP^0$, where the latter is defined as a poset of tuples $(a, f)$, for $a\in \A(\N^R)$ an antichain in $\N^R$ and $f\colon a\rightarrow \{\pm 1\}$ arbitrary, where the order on $\PP^0$ is $(a, f)\leq (b, g)$ if $a<b$, or $a=b$ and $f=g$. Also, $\PP$ is  defined as $\PP^0 - \{(e, -), (e, +)\}$, where $e$ is all-$1$ function on $R$. And we define $\rho \colon \S_\F \rightarrow \PP$ as $\rho(s) = (a_s, f_s)$, where $a_s$ is the antichain of maximal elements in the image $\iota(s)\subseteq \N^R$, and $f_s \colon a_s\rightarrow \{\pm 1\}$ is such that $f_s(\iota(F)) = s(F)$ whenever $\iota(F)\in a_s$.
	
	Now, same as before, we apply \eqref{obs-iota2} to ensure that $\rho$ is well-defined and monotone.  Then, by \Cref{lem-Ky-Fan-generalized}, we get $u\in \SS^n$ such that for $w = (a,f) = \rho(s(u))$ we get $\rk_{\PP^0}(w) = \rk_{\A(\N^R)}(a) \geq (n+3)/2$. Then, for $r\in R$,  we define $h_r(a) = \max \{x(r)~|~x\in a\}$; recall that $a$ is an antichain in $\N^R$, and so $x\in a$ is an $R$-dimensional vector, so $x(r)$ makes sense. As before, $h_r$ is the size of the maximal alternating chain increasing in every $\leq_{r(i)\cdot i}$, for $i\in [d]$, which is witnessed by $u$. Then 
	 
	 \begin{align*}
	 	\rk_{\A(\N^R)} \left( \left\{ \left(h_r(a) \right)_{r\in R} \right\} \right)
	 		&= \prod_{r\in R} h_r(a) - 1 \geq  \rk_{\A(\N^R)}(a)
	 			\geq (n+3)/2,
	 \end{align*}
	 trivially implying 

	 \begin{align*}
	 		\max_{r\in R} h_r(a) &\geq \left((n+1)/2\right)^{1/|R|}
	 			= \left((n+1)/2\right)^{1/2^{d-1}}.
	 \end{align*}
\end{proof}

The fact that Theorems ~\ref{th-kyfan-covers2} and~\ref{th-kyfan-covers3} are asymptotically sharp easily follows from the following statement. Its proof, modulo the statement, is straightforward and is left as an exercise.
\begin{proposition}\label{prop-d-orders}
	For $d, m\geq 1$, let us define:
	\begin{itemize} 
		\item $R$ as the set of functions $r\colon [d]\rightarrow \{\pm1\}$, $|R| = 2^{d-1}$. We consider $R$ to be linearly ordered lexicographically;
		\item $X$ as $[m]^R$, $|X| = m^{2^{d-1}}$;
		\item linear order $\leq_i$ on $X$, for $i=1, \dots, d$, as lexicographic order on $X$, which, on $r$'th coordinate, is either standard linear order on $[m]$ whenever $r(i) = +$, or inverse whenever $r(i) = -$.
	\end{itemize}	
	
	In explicit terms, $\leq_i$ is defined as follows. For $x\neq y \in X = [m]^{R}$, let $r \in R$ be their first, with respect to order on $R$, distinct coordinate. Then $x\leq_i y$ if $x_r < y_r$ and $r(i) = +$ or if $x_r > y_r$ and $r(i) = +$. Alternatively, if $r(i) \cdot x_r < r(i) \cdot y_r$.
	
	Then any sequence in $X$, monotone, in the sense of Theorems ~\ref{th-kyfan-covers2} and~\ref{th-kyfan-covers3}, with respect to all $\leq_i$, for $i=1, \dots, d$, has at most $m \sim |X|^{1/2^{d-1}}$ elements.
\end{proposition}

\bibliographystyle{plain}
\bibliography{bib}

\end{document}